\documentclass{amsproc}%
\usepackage{amsfonts}
\usepackage{amsmath}
\usepackage{amssymb}
\usepackage{graphicx}%
\providecommand{\U}[1]{\protect\rule{.1in}{.1in}}
\theoremstyle{plain}

\newtheorem{corollary}{Corollary}

\newtheorem{lemma}{Lemma}

\newtheorem{remark}{Remark}

\newtheorem{theorem}{Theorem}
\numberwithin{equation}{section}
\begin{document}
\title[Approximate Double Commutants]{Approximate Double Commutants in von Neumann Algebras}
\author{Don Hadwin}
\address{Mathematics Department\\
University of New Hampshire}
\email{don@unh.edu}
\urladdr{http://www.math.unh.edu/\symbol{126}don}
\subjclass[2000]{Primary 46L19; Secondary 46L05}
\keywords{double commutant, approximate double commutant, hyperreflexive}
\dedicatory{Dedicated to Eric Nordgren, a great mathematician and a great friend.}
\begin{abstract}
Richard Kadison showed that not every commutative von Neumann subalgebra of a
factor von Neumann algebra is equal to its relative double commutant. We prove
that every commutative C*-subalgebra of a centrally prime C*-algebra
$\mathcal{B}$ equals its relative approximate double commutant. If
$\mathcal{B}$ is a von Neumann algebra, there is a related distance formula.

\end{abstract}
\maketitle

One of the fundamental results in the theory of von Neumann algebras is von
Neumann's classical \emph{double commutant theorem}, which says that if
$\mathcal{S=S}^{\ast}\subseteq B\left(  H\right)  $, then $\mathcal{S}%
^{\prime\prime}=W^{\ast}\left(  \mathcal{S}\right)  $. In 1978 \cite{H1} the
author proved an asymptotic version of von Neumann's theorem, the
\emph{approximate double commutant theorem}. For the asymptotic version, we
define the \emph{approximate double commutant} of $\mathcal{S}\subseteq
B\left(  H\right)  $, denoted by Appr$\left(  S\right)  ^{\prime\prime}$, to
be the set of all operators $T$ such that
\[
\left\Vert A_{\lambda}T-TA_{\lambda}\right\Vert \rightarrow0
\]
for every bounded net $\left\{  A_{\lambda}\right\}  $ in $B\left(  H\right)
$ for which
\[
\left\Vert A_{\lambda}S-SA_{\lambda}\right\Vert \rightarrow0
\]
for every $S\in\mathcal{S}$. More generally, if $\mathcal{B}$ is a unital
C*-algebra and $\mathcal{S}\subseteq\mathcal{B}$, we define the \emph{relative
approximate double commutant of }$S$\emph{ in }$\mathcal{B}$, denoted by
Appr$\left(  S,\mathcal{B}\right)  ^{\prime\prime}$, in the same way but
insisting that the $T$'s and the $A_{\lambda}$'s be in $\mathcal{B}$. The
approximate double commutant theorem in $B\left(  H\right)  $ \cite{H1} says
that if $\mathcal{S=S}^{\ast},$ then Appr$\left(  \mathcal{S}\right)
^{\prime\prime}=C^{\ast}\left(  \mathcal{S}\right)  $. Moreover, if we
restrict the $\left\{  A_{\lambda}\right\}  $'s to be nets of unitaries or
nets of projections that asymptotically commute with every element of
$\mathcal{S}$, the resulting approximate double commutant is still $C^{\ast
}\left(  \mathcal{S}\right)  $.

A von Neumann algebra $\mathcal{B}$ is \emph{hyperreflexive} if there is a
constant $K\geq1$ such that, for every $T\in B\left(  H\right)  $%
\[
dist\left(  T,\mathcal{B}\right)  \leq K\sup\left\{  \left\Vert
TP-PT\right\Vert :P\in\mathcal{B}^{\prime},P\text{ a projection}\right\}  .
\]
The smallest such $K$ is called the \emph{constant of hyperreflexivity} for
$\mathcal{B}$. The inequality%
\[
\sup\left\{  \left\Vert TP-PT\right\Vert :P\in\mathcal{M}^{\prime},P\text{ a
projection}\right\}  \leq dist\left(  T,\mathcal{M}\right)
\]
is always true. The question of whether every von Neumann algebra is
hyperreflexive is still open and is equivalent to a number of other important
problems in von Neumann algebras (see \cite{HP}). It was proved by the author
\cite{H2} that every unital C*-subalgebra $\mathcal{A}$ of $B\left(  H\right)
$ is approximately hyperreflexive; more precisely, if $T\in B\left(  H\right)
$, then there is a net $\left\{  P_{\lambda}\right\}  $ of projections such
that
\[
\left\Vert AP_{\lambda}-P_{\lambda}A\right\Vert \rightarrow0
\]
for every $A\in\mathcal{A}$, and
\[
dist\left(  T,\mathcal{A}\right)  \leq29\lim_{\lambda}\left\Vert TP_{\lambda
}-P_{\lambda}T\right\Vert .
\]

If we replace the role of $B\left(  H\right)  $ with a factor von Neumann
algebra, then the double commutant theorem fails, even when the subalgebra is
commutative. Suppose $\mathcal{S}$ is a subset of a ring $\mathcal{R}$. We
define the \emph{relative commutant} of $\mathcal{S}$ in $\mathcal{R}$, the
\emph{relative double commutant} of $\mathcal{S}$ in $\mathcal{R}$, and the
\emph{relative triple commutant} of $\mathcal{S}$ in $\mathcal{R}$,
respectively, by%
\[
\left(  \mathcal{S},\mathcal{R}\right)  ^{\prime}=\left\{  T\in\mathcal{R}%
:\forall S\in\mathcal{S},TS=ST\right\}  ,
\]%
\[
\left(  \mathcal{S},\mathcal{R}\right)  ^{\prime\prime}=\left\{
T\in\mathcal{R}:\forall A\in\left(  \mathcal{S},\mathcal{R}\right)  ^{\prime
},TA=AT\right\}  ,
\]
and
\[
\left(  \mathcal{S},\mathcal{R}\right)  ^{\prime\prime\prime}=\left\{
T\in\mathcal{R}:\forall A\in\left(  \mathcal{S},\mathcal{R}\right)
^{\prime\prime},TA=AT\right\}  .
\]
It is clear from general Galois nonsense that
\[
\left(  \mathcal{S},\mathcal{R}\right)  ^{\prime\prime\prime}=\left(
\mathcal{S},\mathcal{R}\right)  ^{\prime}.
\]
Following R. Kadison \cite{Kad} we will say a subring $\mathcal{M}$ of a
unital ring $\mathcal{B}$ is \emph{normal} if
\[
\mathcal{M}=\left(  \mathcal{M},\mathcal{B}\right)  ^{\prime\prime}=\left(
\mathcal{M}^{\prime}\cap\mathcal{B}\right)  ^{\prime}\cap\mathcal{B}\text{.}%
\]
R. Kadison \cite{Kad} proved that if $\mathcal{M}$ is type $I$ von Neumann
subalgebra of a von Neumann algebra $\mathcal{B}$, then $\mathcal{M}$ is
normal in $\mathcal{B}$ if and only if its center $\mathcal{Z}\left(
\mathcal{M}\right)  =\mathcal{M\cap M}^{\prime}$ is normal if and only if
$\mathcal{Z}\left(  \mathcal{M}\right)  $ is an intersection of masas (maximal
abelian selfadjoint subalgebras) of $\mathcal{B}$. See the paper of B. J.
Vowden \cite{V} for more examples. We see that the part of Kadison's result
concerning abelian C*-subalgebras is true in the C*-algebraic setting. We
prove general version for rings, which applies to commutative nonselfadjoint
subalgebras of a C*-algebra or von Neumann algebra.

\bigskip

\begin{lemma}
Suppose $\mathcal{M}$ is a unital abelian subring of a unital ring
$\mathcal{B}$. The following are equivalent:

\begin{enumerate}
\item $\mathcal{M}=\left(  \mathcal{M},\mathcal{B}\right)  ^{\prime\prime}$.

\item $\mathcal{M}$ is an intersection of maximal abelian subrings of
$\mathcal{B}$.

\item $\mathcal{M}$ is an intersection of subrings of the form $\left(
\mathcal{S},\mathcal{B}\right)  ^{\prime}$ for subsets $\mathcal{S}$ of
$\mathcal{B}$.
\end{enumerate}
\end{lemma}

\begin{proof}
First note that every maximal abelian subring $\mathcal{E}$ has the property
that $\mathcal{E}=\left(  \mathcal{E},\mathcal{B}\right)  ^{\prime}$, which
implies $\mathcal{E}=\left(  \mathcal{E},\mathcal{B}\right)  ^{\prime\prime}$
and the implication $\left(  2\right)  \Longrightarrow\left(  3\right)  $. It
is also clear that if $\left\{  \mathcal{S}_{i}:i\in I\right\}  $ is a
collection of nonempty subsets of $\mathcal{B}$, then
\[
\bigcup_{i\in I}\left(  \mathcal{S}_{i},\mathcal{B}\right)  ^{\prime}%
\subseteq\left(  \cap_{i\in I}\mathcal{S}_{i},\mathcal{B}\right)  ^{\prime},
\]
and%
\[
\left(  \cap_{i\in I}\mathcal{S}_{i},\mathcal{B}\right)  ^{\prime\prime
}\subseteq\bigcap_{i\in I}\left(  \mathcal{S}_{i},\mathcal{B}\right)
^{\prime\prime}.
\]
This, and the fact that $\left(  \mathcal{S},\mathcal{B}\right)
^{\prime\prime\prime}=\left(  \mathcal{S},\mathcal{B}\right)  ^{\prime}$
always holds, yields $\left(  3\right)  \Longrightarrow\left(  1\right)  $.

To prove $\left(  1\right)  \Longrightarrow\left(  3\right)  $, suppose
$\left(  1\right)  $ holds, and let $\mathcal{W}$ be a maximal abelian subring
of $\mathcal{B}$ such that $\mathcal{M}\subseteq\mathcal{W}$. For each
$W\in\mathcal{W}\backslash\mathcal{M}$, by $\left(  1\right)  $, there is a
$T_{W}\in\left(  \mathcal{M},\mathcal{B}\right)  ^{\prime}$ such that
$T_{W}W\neq WT_{W}$. Since the ring generated by $\mathcal{M}\cup\left\{
T_{W}\right\}  $ is abelian, it is contained in a maximal abelian subring
$\mathcal{S}_{W}$, and $W\notin\mathcal{S}_{W}$. Hence%
\[
\mathcal{M}=\mathcal{W}\cap\bigcap_{W\in\mathcal{W}\backslash\mathcal{M}%
}\mathcal{S}_{W},
\]
which proves $\left(  2\right)  $ holds.
\end{proof}

\bigskip

If in the statement and proof of the preceding lemma we replace "ring" with
"C*-algebra", we obtain the following result for C*-algebras.

\begin{corollary}
Suppose $\mathcal{M}$ is a unital commutative C*-subalgebra of a unital
C*-algebra $\mathcal{B}$. The following are equivalent:

\begin{enumerate}
\item $\mathcal{M}$ is normal in $\mathcal{B}$.

\item $\mathcal{M}$ is an intersection of masas in $\mathcal{B}$.

\item $\mathcal{M}$ is an intersection of algebras of the form $\left(
\mathcal{S},\mathcal{B}\right)  ^{\prime}$ for subsets $\mathcal{S}$ of
$\mathcal{B}$.
\end{enumerate}
\end{corollary}

\bigskip

We now know that every masa in a C*-algebra is normal. If $\mathcal{M}$ is a
masa in a von Neumann algebra $\mathcal{B}$, then the double commutant theorem
holds even with a distance formula. The proof is a simple adaptation of the
proof of Lemma 3.1 in \cite{R}.

\bigskip

\begin{lemma}
\label{masa}Suppose $\mathcal{M}$ is a masa in a von Neumann algebra
$\mathcal{B}$ and $T\in\mathcal{B}$. Then
\[
dist\left(  T,\mathcal{M}\right)  \leq\sup\left\{  \left\Vert UT-TU\right\Vert
:U=U^{\ast}\in\mathcal{B},U^{2}=1\right\}  =
\]%
\[
2\sup\left\{  \left\Vert TP-PT\right\Vert :P=P^{\ast}=P^{2}\in\mathcal{B}%
\right\}
\]

\end{lemma}

\begin{proof}
Let $R$ denote the right-hand side of the inequality, and let $D$ be the
closed ball in $\mathcal{B}$ centered at $T$ with radius $R$. Suppose
$\mathcal{F}$ is a finite orthogonal set of projections in $\mathcal{M}$ whose
sum is $1$. Let $G\left(  \mathcal{F}\right)  $ be the set of all sums of the
form
\[%
%TCIMACRO{\dsum _{P\in\mathcal{F}}}%
%BeginExpansion
{\displaystyle\sum_{P\in\mathcal{F}}}
%EndExpansion
\lambda_{P}P
\]
with each $\lambda_{P}$ in $\left\{  -1,1\right\}  $. Then $G\left(
\mathcal{F}\right)  $ is a finite group of unitaries and each $U\in G\left(
\mathcal{F}\right)  $ has the form $2Q-1$ with $Q$ a finite sum of elements in
$\mathcal{F}$. Moreover, if $U=2Q-1$,%

\[
2\left\Vert TQ-QT\right\Vert =\left\Vert TU-UT\right\Vert =\left\Vert
T-UTU^{\ast}\right\Vert .
\]
It follows that $UTU^{\ast}\in D$ for every $U\in G\left(  \mathcal{F}\right)
$. Define%
\[
S_{\mathcal{F}}=\frac{1}{cardG\left(  \mathcal{F}\right)  }%
%TCIMACRO{\dsum _{U\in G\left(  F\right)  }}%
%BeginExpansion
{\displaystyle\sum_{U\in G\left(  F\right)  }}
%EndExpansion
UTU^{\ast}.
\]
Since $G\left(  \mathcal{F}\right)  $ is a group, it easily follows that, for
every $U_{0}\in G\left(  \mathcal{F}\right)  $,%
\[
U_{0}S_{\mathcal{F}}U_{0}^{\ast}=S_{\mathcal{F}}.
\]
This implies that $S_{\mathcal{F}}=%
%TCIMACRO{\dsum _{P\in\mathcal{F}}}%
%BeginExpansion
{\displaystyle\sum_{P\in\mathcal{F}}}
%EndExpansion
PTP\in\mathcal{F}^{\prime}=G\left(  \mathcal{F}\right)  ^{\prime}.$ Choose a
subnet $\left\{  S_{\mathcal{F}_{\lambda}}\right\}  $ that converges in the
weak operator topology to $S\in D$. Then $S\in\left(  \mathcal{M}%
,\mathcal{B}\right)  ^{\prime}\cap D$. Since $\left(  \mathcal{M}%
,\mathcal{B}\right)  ^{\prime}=\mathcal{M}$, we conclude
\[
dist\left(  T,\mathcal{M}\right)  \leq\left\Vert T-S\right\Vert \leq R.
\]

\end{proof}

We now address the approximate double commutant relative to a C*-algebra. If
$\mathcal{S}$ is a subset of a C*-algebra $\mathcal{B}$, we know that
Appr$\left(  \mathcal{S},\mathcal{B}\right)  ^{\prime\prime}$ must contain the
center $\mathcal{Z}\left(  \mathcal{B}\right)  =\mathcal{B}\cap\mathcal{B}%
^{\prime}$. Hence if $\mathcal{A}$ is a unital C*-subalgebra of a von Neumann
algebra $\mathcal{B}$, then%
\[
C^{\ast}\left(  \mathcal{A\cup Z}\left(  \mathcal{B}\right)  \right)
\subseteq Appr\left(  \mathcal{A},\mathcal{B}\right)  ^{\prime\prime}.
\]
When $\mathcal{A}$ is commutative, we will prove that equality holds.

The following result is based on S. Macado's generalization \cite{M} of the
Bishop-Stone-Weierstrass theorem. See \cite{Ra} for a beautiful short
elementary proof.

\begin{lemma}
\label{machado}Suppose $\mathcal{W}$ is a unital C*-subalgebra of a
commutative C*-algebra $\mathcal{D}$, and $S=S^{\ast}\in\mathcal{D}$. Then
there are multiplicative linear functionals $\alpha,\beta$ on $\mathcal{D}$
and nets $\left\{  A_{\lambda}\right\}  ,\left\{  B_{\lambda}\right\}
,\left\{  X_{\lambda}\right\}  $ and $\left\{  Y_{\lambda}\right\}  $ in
$\mathcal{D}$ such that

\begin{enumerate}
\item $0\leq X_{\lambda}\leq A_{\lambda}\leq1,0\leq Y_{\lambda}\leq
B_{\lambda}\leq1$,

\item $X_{\lambda}Y_{\lambda}=0,A_{\lambda}X_{\lambda}=X_{\lambda},$
$Y_{\lambda}B_{\lambda}=Y_{\lambda}$,

\item $\left\Vert DX_{\lambda}-\alpha\left(  D\right)  X_{\lambda}\right\Vert
\rightarrow0$ and $\left\Vert DY_{\lambda}-\beta\left(  D\right)  Y_{\lambda
}\right\Vert \rightarrow0$ for every $D\in\mathcal{D}$,

\item $\alpha\left(  A\right)  =\beta\left(  A\right)  $ for every
$A\in\mathcal{W}$,

\item $\alpha\left(  X_{\lambda}\right)  =\beta\left(  Y_{\lambda}\right)  =1$
for every $\lambda$,

\item $\beta\left(  S\right)  -\alpha\left(  S\right)  =2dist\left(
S,\mathcal{W}\right)  $.
\end{enumerate}
\end{lemma}

\begin{proof}
Let $K$ be the maximal ideal space of $\mathcal{D}$ and let $\Gamma
:\mathcal{D}\rightarrow C\left(  K\right)  $ be the Gelfand map, which must be
a $\ast$-isomorphism since $\mathcal{D}$ is a commutative C*-algebra. Let
$g=\Gamma\left(  S\right)  =\Gamma\left(  S^{\ast}\right)  =\bar{g}$ It
follows from Machado's theorem \cite{M} that there is a $\Gamma\left(
\mathcal{W}\right)  $-antisymmetric set $E\subseteq K$ such that%
\[
dist\left(  S,\mathcal{W}\right)  =dist\left(  g,\Gamma\left(  \mathcal{W}%
\right)  \right)  =dist\left(  g|_{E},\Gamma\left(  \mathcal{W}\right)
|_{E}\right)  .
\]
Since $\mathcal{\Gamma}\left(  \mathcal{W}\right)  $ is self-adjoint and $E$
is $\Gamma\left(  \mathcal{W}\right)  $-antisymmetric, every function in
$\Gamma\left(  \mathcal{W}\right)  $ is constant. Hence $dist\left(
g|_{E},\Gamma\left(  \mathcal{W}\right)  |_{E}\right)  $ is the distance from
$g|_{E}$ to the constant functions. It is clear that the closest constant
function to $g|_{E}$ is
\[
\frac{g\left(  \beta\right)  -g\left(  \alpha\right)  }{2},
\]
where $\alpha,\beta\in E,$ $g\left(  \beta\right)  =\max_{x\in E}g\left(
x\right)  $ and $g\left(  \alpha\right)  =\min_{x\in E}g\left(  x\right)  $.
Let $\Lambda$ be the directed set of all pairs $\lambda=\left(  U_{\lambda
},V_{\lambda}\right)  $ of disjoint open sets with $\alpha\in U_{\lambda}$ and
$\beta\in V_{\lambda}$, ordered by $\lambda_{1}\leq\lambda_{2}$ if and only if
$U_{\lambda_{2}}\subseteq U_{\lambda_{1}}$ and $V_{\lambda_{2}}\subseteq
V_{\lambda_{1}}$. For each $\lambda\in\Lambda$ choose continuous functions
$r_{\lambda},s_{\lambda},t_{\lambda},u_{\lambda}:K\rightarrow\left[
0,1\right]  $ such that

\begin{enumerate}
\item[a.] $r_{\lambda}\left(  \alpha\right)  =t_{\lambda}\left(  \beta\right)
=1$,

\item[b.] $0\leq r_{\lambda}=r_{\lambda}s_{\lambda}\leq s_{\lambda}\leq1$,

\item[c.] $0\leq t_{\lambda}=t_{\lambda}u_{\lambda}\leq u_{\lambda}\leq1,$

\item[d.] supp$s_{\lambda}\subseteq U_{\lambda}$ and supp$u_{\lambda}\subseteq
V_{\lambda}$.
\end{enumerate}

If we choose $A_{\lambda},B_{\lambda}X_{\lambda},Y_{\lambda}\in\mathcal{A}$
such that $\Gamma\left(  X_{\lambda}\right)  =r_{\lambda},$ $\Gamma\left(
A_{\lambda}\right)  =s_{\lambda}$, $\Gamma\left(  Y_{\lambda}\right)
=t_{\lambda}$, and $\Gamma\left(  B_{\lambda}\right)  =u_{\lambda}$, then
statements $\left(  1\right)  $-$\left(  6\right)  $ are clear.
\end{proof}

\bigskip

A C*-algebra $\mathcal{B}$ is \emph{primitive} if it has a faithful
irreducible representation. A C*-algebra $\mathcal{B}$ is \emph{prime} if, for
every $x,y\in\mathcal{B}$, we have
\[
x\mathcal{B}y=\left\{  0\right\}  \Longrightarrow x=0\text{ or }y=0.
\]
Every primitive C*-algebra is prime, and it was proved by Dixmier \cite{D}
that every separable prime C*-algebra is primitive. N. Weaver \cite{W} gave an
example of a nonseparable prime C*-algebra that is not primitive.

We define $\mathcal{B}$ to be \emph{centrally prime} if, whenever
$x,y\in\mathcal{B},$ $0\leq x,y\leq1$ and $x\mathcal{B}y=\left\{  0\right\}
$, there is an $e\in\mathcal{Z}\left(  \mathcal{B}\right)  $ such that $x\leq
e\leq1$ and $y\leq1-e\leq1$. The centrally prime algebras in include the prime
ones, von Neumann algebras, and $%
%TCIMACRO{\dprod _{i\in I}}%
%BeginExpansion
{\displaystyle\prod_{i\in I}}
%EndExpansion
\mathcal{B}_{i}/%
%TCIMACRO{\dsum _{i\in I}}%
%BeginExpansion
{\displaystyle\sum_{i\in I}}
%EndExpansion
\mathcal{B}_{i}$ or a C*-ultraproduct $%
%TCIMACRO{\dprod \limits_{i\in I}^{\alpha}}%
%BeginExpansion
{\displaystyle\prod\limits_{i\in I}^{\alpha}}
%EndExpansion
\mathcal{B}_{i}$ when $\left\{  \mathcal{B}_{i}:i\in I\right\}  $ is a
collection of unital primitive C*-algebras (see the proof of Theorem
\ref{dcAppl}) .

We characterize $Appr\left(  \mathcal{A},\mathcal{B}\right)  ^{\prime\prime}$
for every commutative C*-subalgebra $\mathcal{A}$ of a centrally prime
C*-algebra $\mathcal{B}$, and we show that there is a distance formula for
every commutative unital C*-subalgebra if and only if every masa in
$\mathcal{B}$ has a distance formula. In particular, when $\mathcal{B}$ is a
von Neumann algebra, we obtain a distance formula.

\begin{remark}
\label{distform}Here is a useful comment on distance formulas. If
$\mathcal{B}$ is a unital C*-algebra and $\mathcal{S}=\mathcal{S}^{\ast
}\subseteq\mathcal{B}$, then $\left(  \mathcal{S},\mathcal{B}\right)
^{\prime}$ is a unital C*-algebra, so, by the Russo-Dye theorem, the closed
unit ball of $\left(  \mathcal{S},\mathcal{B}\right)  ^{\prime}$ is the
norm-closed convex hull of the set of unitary elements in $\left(
\mathcal{S},\mathcal{B}\right)  ^{\prime}$. Hence, for any $T\in\mathcal{B}$,%
\[
\sup\left\{  \left\Vert TW-WT\right\Vert :W\in\left(  \mathcal{S}%
,\mathcal{B}\right)  ^{\prime},\left\Vert W\right\Vert \leq1\right\}  =
\]%
\[
\sup\left\{  \left\Vert TU-UT\right\Vert :U\in\left(  \mathcal{S}%
,\mathcal{B}\right)  ^{\prime},U\text{ is unitary}\right\}  .
\]
A similar result holds in the approximate case. Suppose $\left(  \Lambda
,\leq\right)  $ is a directed set. Then $%
%TCIMACRO{\dprod _{\lambda\in\Lambda}}%
%BeginExpansion
{\displaystyle\prod_{\lambda\in\Lambda}}
%EndExpansion
\mathcal{B}$ is a unital C*-algebra and the set
\[
\mathcal{E=}\left\{  \left\{  W_{\lambda}\right\}  \in%
%TCIMACRO{\dprod _{\lambda\in\Lambda}}%
%BeginExpansion
{\displaystyle\prod_{\lambda\in\Lambda}}
%EndExpansion
\mathcal{B}:\forall S\in\mathcal{S},\lim_{\lambda}\left\Vert W_{\lambda
}S-SW_{\lambda}\right\Vert =0\right\}
\]
is a unital C*-algebra and is the closed convex hull of its unitary group.
Hence%
\[
\sup\left\{  \limsup_{\lambda}\left\Vert TW_{\lambda}-W_{\lambda}T\right\Vert
:W=\left\{  W_{\lambda}\right\}  \in\mathcal{E},\left\Vert W\right\Vert
\leq1\right\}  =
\]%
\[
\sup\left\{  \limsup_{\lambda}\left\Vert TU_{\lambda}-U_{\lambda}T\right\Vert
:U=\left\{  U_{\lambda}\right\}  \in\mathcal{E},U\text{ is unitary}\right\}
.
\]

\end{remark}

\bigskip

\begin{theorem}
\label{abelian}Suppose $\mathcal{B}$ is a centrally prime unital C*-algebra
and $\mathcal{Z}\left(  \mathcal{B}\right)  \subseteq\mathcal{W}%
\subseteq\mathcal{D}$ are unital commutative C*-subalgebras of $\mathcal{B}$.
Suppose $S=S^{\ast}\in\mathcal{D}$. Then there is a net $\left\{  W_{\lambda
}\right\}  $ in $\mathcal{B}$ such that

\begin{enumerate}
\item $W_{\lambda}$ is unitary for every $\lambda$,

\item $\lim_{\lambda}\left\Vert AW_{\lambda}-W_{\lambda}A\right\Vert =0$ for
every $A\in\mathcal{W}$,

\item $\lim_{\lambda}$ $\left\Vert SW_{\lambda}-W_{\lambda}S\right\Vert
=2dist\left(  S,\mathcal{W}\right)  $.
\end{enumerate}

Moreover, if $\mathcal{B}$ is a von Neumann algebra,, then there is a sequence
$\left\{  P_{n}\right\}  $ of projections in $\mathcal{B}$ such that

\begin{enumerate}
\item[4.] $\lim_{\lambda}\left\Vert AP_{\lambda}-P_{\lambda}A\right\Vert =0$
for every $A\in\mathcal{W}$,

\item[5.] $\lim_{\lambda}$ $\left\Vert SP_{\lambda}-P_{\lambda}S\right\Vert
=dist\left(  S,\mathcal{W}\right)  $.
\end{enumerate}
\end{theorem}

\begin{proof}
Let $\mathcal{W}=C^{\ast}\left(  \mathcal{A}\cup\mathcal{Z}\left(
\mathcal{B}\right)  \cup\left\{  S\right\}  \right)  ,$ $\mathcal{D}=C^{\ast
}\left(  \mathcal{A}\cup\mathcal{Z}\left(  \mathcal{B}\right)  \cup\left\{
S\right\}  \right)  $. Now choose $\alpha,\beta$ and nets $\left\{
A_{\lambda}\right\}  ,\left\{  B_{\lambda}\right\}  ,\left\{  X_{\lambda
}\right\}  $ and $\left\{  Y_{\lambda}\right\}  $ in $\mathcal{D}$ as in Lemma
\ref{machado}. We first show that $X_{\lambda}\mathcal{B}Y_{\lambda}%
\neq\left\{  0\right\}  $; otherwise, since $\mathcal{B}$ is centrally prime,
there is an $e\in\mathcal{Z}\left(  \mathcal{B}\right)  $ such that
$X_{\lambda}\leq e\leq1$ and $Y_{\lambda}\leq1-e\leq1.$ Hence $\alpha\left(
e\right)  =1$ and $\beta\left(  1-e\right)  =0$, or $\beta\left(  e\right)
=0$. However, $e\in\mathcal{W}$ and, by part $\left(  4\right)  $ of Lemma
\ref{machado}, we get $\alpha\left(  e\right)  =\beta\left(  e\right)  $. This
contradiction shows that $X_{\lambda}\mathcal{B}Y_{\lambda}\neq\left\{
0\right\}  .$ Hence there is a $C_{\lambda}\in\mathcal{B}$ such that
$\left\Vert X_{\lambda}C_{\lambda}Y_{\lambda}\right\Vert =1.$ Define
$W_{\lambda}=X_{\lambda}C_{\lambda}Y_{\lambda}=A_{\lambda}W_{\lambda
}=W_{\lambda}B_{\lambda}$. Lemma \ref{machado} implies that, for every
$D\in\mathcal{D}$,%
\[
\left\Vert DW_{\lambda}-\alpha\left(  D\right)  W_{\lambda}\right\Vert
=\left\Vert DA_{\lambda}W_{\lambda}-\alpha\left(  D\right)  A_{\lambda
}W_{\lambda}\right\Vert \leq\left\Vert \left[  D-\alpha\left(  D\right)
\right]  A_{\lambda}\right\Vert \left\Vert W_{\lambda}\right\Vert
\rightarrow0,
\]
and
\[
\left\Vert W_{\lambda}D-\beta\left(  D\right)  W_{\lambda}\right\Vert
=\left\Vert W_{\lambda}B_{\lambda}D-\beta\left(  D\right)  W_{\lambda
}B_{\lambda}\right\Vert \leq\left\Vert W_{\lambda}\right\Vert \left\Vert
B_{\lambda}\left[  D-\alpha\left(  D\right)  \right]  \right\Vert
\rightarrow0.
\]
Since $\alpha\left(  A\right)  =\beta\left(  A\right)  $ for every
$A\in\mathcal{W}$, it follows that $\left\Vert AW_{\lambda}-W_{\lambda
}A\right\Vert \rightarrow0$. It also follows that
\[
\lim_{\lambda}\left\Vert W_{\lambda}S-SW_{\lambda}\right\Vert =\lim_{\lambda
}\left\vert \beta\left(  S\right)  -\alpha\left(  S\right)  \right\vert
\left\Vert W_{\lambda}\right\Vert =\left\vert \beta\left(  S\right)
-\alpha\left(  S\right)  \right\vert =2dist\left(  S,\mathcal{W}\right)  .
\]
We now appeal to Remark \ref{distform} to replace the net $\left\{
W_{\lambda}\right\}  $ with a net of unitaries.

Now suppose $\mathcal{B}$ is a von Neumann algebra. Once we get $X_{\lambda
}\mathcal{B}Y_{\lambda}\neq0$ we know that there is a partial isometry
$V_{\lambda}$ in $\mathcal{B}$ whose final space is contained in the closure
of $ranX_{\lambda}$ and whose initial space is contained in $\left(  \ker
Y_{\lambda}\right)  ^{\perp}$. Then $\left(  3\right)  $ holds with $\left\{
W_{\lambda}\right\}  $ replaced with $\left\{  V_{\lambda}\right\}  .$ Also,
$V_{\lambda}^{2}=0$ (since $X_{\lambda}Y_{\lambda}=0$), so $P_{\lambda}%
=\frac{1}{2}\left(  V_{\lambda}+V_{\lambda}^{\ast}+V_{\lambda}V_{\lambda
}^{\ast}+V_{\lambda}^{\ast}V_{\lambda}\right)  $ is a projection. Using the
above arguments gives us%
\[
\left\Vert DV_{\lambda}^{\ast}V_{\lambda}-\beta\left(  D\right)  V_{\lambda
}^{\ast}V_{\lambda}\right\Vert \rightarrow0,\left\Vert V_{\lambda}^{\ast
}V_{\lambda}D-\beta\left(  D\right)  V_{\lambda}^{\ast}V_{\lambda}\right\Vert
\rightarrow0
\]
and
\[
\left\Vert DV_{\lambda}V_{\lambda}^{\ast}-\alpha\left(  D\right)  V_{\lambda
}V_{\lambda}^{\ast}\right\Vert \rightarrow0,\left\Vert V_{\lambda}V_{\lambda
}^{\ast}D-\alpha\left(  D\right)  V_{\lambda}V_{\lambda}^{\ast}\right\Vert
\rightarrow0,
\]
which implies
\[
\left\Vert DV_{\lambda}^{\ast}V_{\lambda}-V_{\lambda}^{\ast}V_{\lambda
}D+DV_{\lambda}V_{\lambda}^{\ast}-V_{\lambda}V_{\lambda}^{\ast}D\right\Vert
\rightarrow0
\]
for every $D\in\mathcal{B}$. Thus
\[
\lim_{\lambda}\left\Vert SP_{\lambda}-P_{\lambda}S\right\Vert =\frac{1}{2}%
\lim\left\Vert \left(  \alpha\left(  S\right)  V_{\lambda}-V_{\lambda}%
\beta\left(  S\right)  \right)  +\left(  \beta\left(  S\right)  V_{\lambda
}^{\ast}-V_{\lambda}^{\ast}\alpha\left(  S\right)  \right)  \right\Vert =
\]%
\[
\lim_{\lambda}\frac{1}{2}\left\vert \beta\left(  S\right)  -\alpha\left(
S\right)  \right\vert \left\Vert V_{\lambda}^{\ast}-V_{\lambda}\right\Vert
=\frac{1}{2}\left\vert \beta\left(  S\right)  -\alpha\left(  S\right)
\right\vert =dist\left(  S,\mathcal{W}\right)  ,
\]
since $\left\Vert V_{\lambda}^{\ast}-V_{\lambda}\right\Vert =1$ for every
$\lambda$.
\end{proof}

\bigskip

\begin{theorem}
Suppose $\mathcal{A}$ is a unital commutative C*-subalgebra of a centrally
prime unital C*-algebra $\mathcal{B}$. Then%
\[
Appr\left(  \mathcal{A},\mathcal{B}\right)  ^{\prime\prime}=C^{\ast}\left(
\mathcal{A}\cup\mathcal{Z}\left(  \mathcal{B}\right)  \right)  .
\]
Hence $\mathcal{A}$ is normal if and only if $\mathcal{Z}\left(
\mathcal{B}\right)  \subseteq\mathcal{A}$.
\end{theorem}

\begin{proof}
It is clear that $\mathcal{W}=$ $C^{\ast}\left(  \mathcal{A}\cup
\mathcal{Z}\left(  \mathcal{B}\right)  \right)  \subseteq Appr\left(
\mathcal{A},\mathcal{B}\right)  ^{\prime\prime}$. Choose a masa $\mathcal{D}$
of $\mathcal{B}$ with $\mathcal{A}\subseteq\mathcal{D}$. Then%
\[
\mathcal{W}\subseteq Appr\left(  \mathcal{A},\mathcal{B}\right)
^{\prime\prime}\subseteq Appr\left(  \mathcal{D},\mathcal{B}\right)
^{\prime\prime}=\mathcal{D}.
\]
If we choose $S=S^{\ast}\in Appr\left(  \mathcal{A},\mathcal{B}\right)
^{\prime\prime}$ and apply Theorem \ref{abelian} we see that $S\in\mathcal{W}%
$. Since $Appr\left(  \mathcal{A},\mathcal{B}\right)  ^{\prime\prime}$ is a
C*-algebra, we have proved that $Appr\left(  \mathcal{A},\mathcal{B}\right)
^{\prime\prime}\subseteq\mathcal{W}$.\bigskip
\end{proof}

In the von Neumann algebra setting, we get a distance formula. We have not
tried to get the best constant.

\begin{theorem}
\label{main}Suppose $\mathcal{A}$ is a unital commutative C*-subalgebra of a
von Neumann algebra $\mathcal{B}$ and $T\in\mathcal{B}$. Then there is a net
$\left\{  P_{\lambda}\right\}  $ of projections in $\mathcal{B}$ such that,

\begin{enumerate}
\item for every $A\in\mathcal{A}$,%
\[
\left\Vert AP_{\lambda}-P_{\lambda}A\right\Vert \rightarrow0,
\]
and

\item
\[
dist\left(  T,C^{\ast}\left(  \mathcal{A}\cup\mathcal{Z}\left(  \mathcal{B}%
\right)  \right)  \right)  \leq10\lim_{\lambda}\left\Vert TP_{\lambda
}-P_{\lambda}T\right\Vert .
\]

\end{enumerate}
\end{theorem}

\begin{proof}
Let $\mathcal{W}=C^{\ast}\left(  \mathcal{A}\cup\mathcal{Z}\left(
\mathcal{B}\right)  \right)  $. We define the seminorm $\Delta$ on
$\mathcal{B}$ by $\Delta\left(  V\right)  $ to be the supremum of
$\lim_{\lambda}\left\Vert VP_{\lambda}-P_{\lambda}V\right\Vert $ taken over
all nets $\left\{  P_{\lambda}\right\}  $ of projections in $\mathcal{B}$ for
which $\left\Vert AP_{\lambda}-P_{\lambda}A\right\Vert \rightarrow0$ for every
$A\in\mathcal{A}$ and $\lim_{\lambda}\left\Vert VP_{\lambda}-P_{\lambda
}V\right\Vert $ exists. Let $\mathcal{D}$ be a masa in $\mathcal{B}$ such that
$\mathcal{W}\subseteq\mathcal{D}$.

We first assume $T=T^{\ast}$. It follows from Lemma \ref{masa} that there is
an $S\in\mathcal{D}$ such that%
\[
\left\Vert S-T\right\Vert \leq2\sup\left\{  \left\Vert TP-PT\right\Vert
:P=P^{\ast}=P^{2}\in\mathcal{D}\right\}  \leq2\Delta\left(  T\right)  \text{.}%
\]
If we apply Theorem \ref{abelian}, we obtain a net $\left\{  P_{\lambda
}\right\}  $ of projections in $\mathcal{B}$ such that%

\[
\lim_{\lambda}\left\Vert WP_{\lambda}-P_{\lambda}W\right\Vert =0
\]
for every $W\in\mathcal{W}$, and such that
\[
\lim_{\lambda}\left\Vert P_{\lambda}S-SP_{\lambda}\right\Vert =dist\left(
S,\mathcal{W}\right)  .
\]
It follows that
\[
dist\left(  T,\mathcal{W}\right)  \leq dist\left(  S,\mathcal{W}\right)
+\left\Vert S-T\right\Vert \leq\Delta\left(  S\right)  +2\Delta\left(
T\right)  \leq
\]
and%
\[
\Delta\left(  S-T\right)  +\Delta\left(  T\right)  +2\Delta\left(  T\right)
\leq\left\Vert S-T\right\Vert +3\Delta\left(  T\right)  \leq5\Delta\left(
T\right)  .
\]

whenever $T=T^{\ast}.$

For the general case,%
\[
dist\left(  T,\mathcal{A}\right)  \leq dist\left(  \operatorname{Re}%
T,\mathcal{A}\right)  +dist\left(  \operatorname{Im}T,\mathcal{A}\right)  \leq
\]%
\[
5\Delta\left(  \operatorname{Re}T\right)  +5\Delta\left(  \operatorname{Im}%
T\right)  \leq5\left[  \frac{1}{2}\Delta\left(  T+T^{\ast}\right)  +\frac
{1}{2}\Delta\left(  T-T^{\ast}\right)  \right]  \leq
\]%
\[
5\left[  \Delta\left(  T\right)  +\Delta\left(  T^{\ast}\right)  \right]
=10\Delta\left(  T\right)  ,
\]
since $\Delta\left(  T\right)  =\Delta\left(  T^{\ast}\right)  $.
\end{proof}

\bigskip

\begin{corollary}
If $\mathcal{B}$ is a centrally prime C*-algebra with trivial center, e.g., a
factor von Neumann algebra or the Calkin algebra, then $\mathcal{A}%
=Appr\left(  \mathcal{A},\mathcal{B}\right)  ^{\prime\prime}$ for every
commutative unital C*-subalgebra $\mathcal{A}$ of $\mathcal{B}$.\bigskip
\end{corollary}

In some cases our results yield information on relative double commutants.

\bigskip

\begin{theorem}
\label{dcAppl}Suppose $\left\{  \mathcal{B}_{n}\right\}  $ is a sequence of
primitive C*-algebras and $\mathcal{B}=%
%TCIMACRO{\dprod _{n\geq1}}%
%BeginExpansion
{\displaystyle\prod_{n\geq1}}
%EndExpansion
\mathcal{B}_{n}/%
%TCIMACRO{\dsum _{n\geq1}^{\oplus}}%
%BeginExpansion
{\displaystyle\sum_{n\geq1}^{\oplus}}
%EndExpansion
\mathcal{B}_{n}$. If $\mathcal{A}$ is a separable unital C*-subalgebra of
$\mathcal{B}$, then%
\[
\left(  \mathcal{A},\mathcal{B}\right)  ^{\prime\prime}=C^{\ast}\left(
\mathcal{A}\cup\mathcal{Z}\left(  \mathcal{B}\right)  \right)  ,
\]
i.e., $C^{\ast}\left(  \mathcal{A}\cup\mathcal{Z}\left(  \mathcal{B}\right)
\right)  $ is normal.
\end{theorem}

\begin{proof}
We first show that $\mathcal{B}$ is centrally prime. Since each $\mathcal{B}%
_{n}$ is primitive, we can assume, for each $n\in\mathbb{N}$, that there is a
Hilbert space $H_{n}$ such that $\mathcal{B}_{n}$ is an irreducible unital
C*-subalgebra of $B\left(  H_{n}\right)  $. Suppose $A,B\in\mathcal{B}$,
$0\leq A,B\leq1$ and $A\mathcal{B}B=0$. We can lift $A,B,$ respectively to a
sequences $\left\{  A_{n}\right\}  ,\left\{  B_{n}\right\}  $ in $%
%TCIMACRO{\dprod _{n\geq1}}%
%BeginExpansion
{\displaystyle\prod_{n\geq1}}
%EndExpansion
\mathcal{B}_{n}$. Hence, for every bounded sequence $\left\{  T_{n}\right\}
\in%
%TCIMACRO{\dprod _{n\geq1}}%
%BeginExpansion
{\displaystyle\prod_{n\geq1}}
%EndExpansion
\mathcal{B}_{n}$, we have%
\[
\lim_{n\rightarrow\infty}\left\Vert A_{n}T_{n}B_{n}\right\Vert =0.
\]
Choose unit vectors $e_{n},f_{n}\in H_{n}$ so that $\left\Vert A_{n}%
e_{n}\right\Vert \geq\left\Vert A_{n}\right\Vert /2$ and $\left\Vert
B_{n}f_{n}\right\Vert \geq\left\Vert B_{n}\right\Vert /2$. It follwos from the
irreducibility of $\mathcal{B}_{n}$ and Kadison's transitivity theorem
\cite{KR} that there is a $T_{n}\in\mathcal{B}_{n}$ such that $\left\Vert
T_{n}\right\Vert =1$ and $T_{n}B_{n}f_{n}=\left\Vert B_{n}f_{n}\right\Vert
e_{n}$. It follows that
\[
0=\lim_{n\rightarrow\infty}\left\Vert A_{n}T_{n}B_{n}\right\Vert \geq
\lim_{n\rightarrow\infty}\left\Vert A_{n}T_{n}B_{n}f_{n}\right\Vert \geq
\lim\frac{1}{4}\left\Vert A_{n}\right\Vert \left\Vert B_{n}\right\Vert .
\]
Hence%
\[
\lim_{n\rightarrow\infty}\min\left(  \left\Vert A_{n}\right\Vert ,\left\Vert
B_{n}\right\Vert \right)  ^{2}\leq\lim_{n\rightarrow\infty}\left\Vert
A_{n}\right\Vert \left\Vert B_{n}\right\Vert =0.
\]
For each $n\in\mathbb{N}$ we define%
\[
P_{n}=\left\{
\begin{array}
[c]{cc}%
1 & \text{if }\left\Vert B_{n}\right\Vert \leq\left\Vert A_{n}\right\Vert \\
0 & \text{if }\left\Vert A_{n}\right\Vert <\left\Vert B_{n}\right\Vert
\end{array}
\right.  .
\]
Then $\left\{  P_{n}\right\}  $ is in the center of $%
%TCIMACRO{\dprod _{n\geq1}}%
%BeginExpansion
{\displaystyle\prod_{n\geq1}}
%EndExpansion
\mathcal{B}_{n}$ and
\[
\lim_{n\rightarrow\infty}\left\Vert P_{n}B_{n}\right\Vert =\lim_{n\rightarrow
\infty}\left\Vert \left(  1-P_{n}\right)  A_{n}\right\Vert =0.
\]
If $P$ is the image of $\left\{  P_{n}\right\}  $ in the quotient
$\mathcal{B}$, then $P$ is a central projection and $PA=P$ and $\left(
1-P\right)  B=B$. Hence $\mathcal{B}$ is centrally prime. So it follows that
\[
appr\left(  \mathcal{A},\mathcal{B}\right)  ^{\prime\prime}=C^{\ast}\left(
\mathcal{A}\cup\mathcal{Z}\left(  \mathcal{B}\right)  \right)  .
\]

The proof will be completed with proof of the following claim: If
$\mathcal{S}$ is a norm-separable subset of $\mathcal{B}$, then
\[
appr\left(  \mathcal{S},\mathcal{B}\right)  ^{\prime\prime}=\left(
\mathcal{S},\mathcal{B}\right)  ^{\prime\prime}.
\]
It is clear from considering constant sequences that the inclusion
$appr\left(  \mathcal{S},\mathcal{B}\right)  ^{\prime\prime}\subseteq\left(
\mathcal{S},\mathcal{B}\right)  ^{\prime\prime}$ holds for every unital
C*-algebra $\mathcal{B}$. To prove the reverse inclusion, suppose $T\notin
appr\left(  \mathcal{S},\mathcal{B}\right)  ^{\prime\prime}$. Then there is
and $\varepsilon>0$ and a net $\left\{  A_{\lambda}\right\}  $ in
$\mathcal{B}$ such that $\left\Vert A_{\lambda}S-SA_{\lambda}\right\Vert
\rightarrow0$ for every $S\in\mathcal{S}$, and such that $\left\Vert
A_{\lambda}T-TA_{\lambda}\right\Vert \geq\varepsilon$ for every $\lambda$. Let
$\mathcal{S}_{0}=\left\{  S_{1},S_{2},\ldots\right\}  $ be a dense subset of
$\mathcal{S}$. We can lift each $S_{n}$ to $\left\{  S_{n}\left(  j\right)
\right\}  _{j\geq1}\in%
%TCIMACRO{\dprod _{k\geq1}}%
%BeginExpansion
{\displaystyle\prod_{k\geq1}}
%EndExpansion
\mathcal{B}_{k}$ and lift $T$ to $\left\{  T\left(  j\right)  \right\}
_{j\geq1}$. It follows that, for every $n\in\mathbb{N}$, there is an $A_{n}%
\in\mathcal{B}$ with $\left\Vert A_{n}\right\Vert =1$ such that

\begin{enumerate}
\item[a.] $\left\Vert A_{n}S_{k}-S_{k}A_{n}\right\Vert <1/n$ for $1\leq k\leq
n,$

\item[b.] $\left\Vert A_{n}T-TA_{n}\right\Vert >\varepsilon/2$.
\end{enumerate}

Note that if $B\in\mathcal{B}$ lifts to $\left\{  B\left(  j\right)  \right\}
_{j\geq1}\in%
%TCIMACRO{\dprod _{k\geq1}}%
%BeginExpansion
{\displaystyle\prod_{k\geq1}}
%EndExpansion
\mathcal{B}_{k}$, then $\left\Vert B\right\Vert =\limsup_{j\rightarrow\infty
}\left\Vert B\left(  j\right)  \right\Vert $. If we lift each $A_{n}$ to
$\left\{  A_{n}\left(  j\right)  \right\}  ,$ it follows that we can find an
arbitrarily large $j_{n}\in\mathbb{N}$ such that $\left\Vert A_{n}\left(
j_{n}\right)  S_{k}\left(  j_{n}\right)  -S_{k}\left(  j_{n}\right)
A_{n}\left(  j_{n}\right)  \right\Vert <1/n$ for $1\leq k\leq n$ and
$\left\Vert A_{n}\left(  j_{n}\right)  T\left(  j_{n}\right)  -T\left(
j_{n}\right)  A_{n}\left(  j_{n}\right)  \right\Vert >\varepsilon/2$. Since
$j_{n}$ can be chosen to be arbitrarily large, we can choose $\left\{
j_{n}\right\}  $ so that $j_{1}<j_{2}<\cdots$. We now define $A\in\mathcal{B}$
by defining%
\[
A\left(  j\right)  =\left\{
\begin{array}
[c]{cc}%
A_{n}\left(  j_{n}\right)  & \text{if }j=j_{n}\text{ for some }n\geq1\\
0 & \text{otherwise}%
\end{array}
\right.  .
\]
We see that $AS_{k}=S_{k}A$ for all $k\geq1$ and $\left\Vert AT-TA\right\Vert
\geq\varepsilon/2$. Hence $T\notin\left(  \mathcal{S},\mathcal{B}\right)
^{\prime\prime}$.
\end{proof}

\bigskip

We conclude with some questions.

\begin{enumerate}
\item If $\mathcal{M}$ is a normal von Neumann subalgebra of a factor von
Neumann algebra $\mathcal{B}$, is there a constant $K\geq1$ such that, for
every $T\in\mathcal{B},$%
\[
dist\left(  T,\mathcal{M}\right)  \leq K\sup\left\{  \left\Vert
TP-PT\right\Vert :P=P^{2}=P^{\ast}\in\mathcal{M}^{\prime}\cap\mathcal{B}%
\right\}  ?
\]
When $\mathcal{B}=B\left(  H\right)  ,$ this question is equivalent to
Kadison's similarity problem. What about factors not of type $I$?

\item Is there an analog of Theorem \ref{main} for arbitrary C*-subalgebras of
a factor von Neumann algebra?

\item It seems likely that a version of parts $\left(  4\right)  $ and
$\left(  5\right)  $ of Theorem \ref{abelian} might hold under assumptions
weaker than $\mathcal{B}$ being a von Neumann algebra. Is it true when
$\mathcal{B}$ has real-rank zero? What if we include nuclear and simple? The
key is getting the partial isometries $V_{\lambda}$ in the proof of Theorem
\ref{abelian}. When does a unital C*-algebra $\mathcal{B}$ have the property
that whenever $X,Y,A,B\geq0$ are in $\mathcal{B}$ and $AX=X,$ $BY=Y,$ $AB=0$
and $X\mathcal{B}Y\neq\left\{  0\right\}  $, there is a nonzero partial
isometry $V\in\mathcal{B}$ such that $AV=VB=V$?
\end{enumerate}


\begin{thebibliography}{99}                                                                                               %


\bibitem {C}Choda, Marie, A condition to construct a full $II_{1}$-factor with
an application to approximate normalcy. Math. Japon. 28 (1983) 383--398.

\bibitem {D}Dixmier, J. Sur les C -alg\`{e}bres, Bull. Soc. Math. France 88
(1960) 95--112.

\bibitem {H1}Hadwin, Don, An asymptotic double commutant theorem for
C*-algebras, Trans. Amer. Math. Soc. 244 (1978) 273--297.

\bibitem {H2}Hadwin, Don, Approximately hyperreflexive algebras, J. Operator
Theory 28 (1992) 51--64.

\bibitem {H3}Hadwin, Don, Continuity modulo sets of measure zero, Math.
Balkanica (N.S.) 3 (1989) 430--433.

\bibitem {HP}Hadwin, Don; Paulsen, Vern I., Two reformulations of Kadison's
similarity problem, J. Operator Theory 55 (2006) 3--16

\bibitem {J}Jolissaint, Paul, Operator algebras related to Thompson's group
F., J. Aust. Math. Soc. 79 (2005) 231--241.

\bibitem {Kad}Kadison, Richard V. Normalcy in operator algebras. Duke Math. J.
29 1962 459--464.

\bibitem {KR}Kadison, Richard V. and Ringrose, J. R., Fundamentals of the
theory of operator algebras, Vol. II, New York: Harcourt, 1986.

\bibitem {Kas}Kasparov, G. G., The operator K-functor and extensions of
C*-algebras, Math. USSR-Isv. 16 (1981) 513-572.

\bibitem {M}Machado, Silvio, On Bishop's generalization of the
Weierstrass-Stone theorem, Indag. Math. 39 (1977) 218--224.

\bibitem {Ra}Ransford, T. J., A short elementary proof of the
Bishop-Stone-Weierstrass theorem. Math. Proc. Cambridge Philos. Soc. 96
(1984), no. 2, 309--311.

\bibitem {R}Rosenoer, Shlomo, Distance estimates for von Neumann algebras.
Proc. Amer. Math. Soc. 86 (1982) 248--252.

\bibitem {V}Vowden, B. J., Normalcy in von Neumann algebras, Proc. London
Math. Soc. (3) 27 (1973) 88--100.

\bibitem {W}Weaver, Nik, A prime C*-algebra that is not primitive, J. Funct.
Anal. 203 (2003) 356--361.
\end{thebibliography}
\end{document}